\documentclass[10pt,a4paper]{article}
\usepackage{amsmath}
\usepackage{amsfonts}
\usepackage{amssymb}
\usepackage{graphicx}
\usepackage{amsthm}
\usepackage[all]{xy}
\usepackage{comment}
\usepackage[subnum]{cases}

\newcounter{oftheorem}[subsection]
\newenvironment{mytheorem}[1]%
{\begin{trivlist}
     \renewcommand{\theoftheorem}{#1~\thesection.\arabic{oftheorem}}
     \refstepcounter{oftheorem}
     \item[\hspace{\labelsep}\bf\thesection.\arabic{oftheorem} #1]}%
{\end{trivlist}}
\newenvironment{definition}{\begin{mytheorem}{Definition}\it}{\end{mytheorem}}
\newenvironment{example}{\begin{mytheorem}{Example}\it}{\end{mytheorem}}

\newenvironment{proposition}{\begin{mytheorem}{Proposition}\it}{\end{mytheorem}}
\newenvironment{theorem}{\begin{mytheorem}{Theorem}\it}{\end{mytheorem}}

\newenvironment{remark}{\begin{mytheorem}{Remark}}{\end{mytheorem}}
\newenvironment{lemma}{\begin{mytheorem}{Lemma}}{\end{mytheorem}}

\begin{document}

\title{Planar vector fields in the kernel of a 1--form}

\author{Stavros Anastassiou\\
Department of Mathematics\\
University of West Macedonia\\ GR-52100 Kastoria, Greece\\
sanastassiou@gmail.com\\
https://orcid.org/0000-0001-5033-4699}

\maketitle

\begin{abstract}
We classify, up to a natural equivalence relation, vector fields of the plane which belong to the kernel of a 1--form. This form can be closed, in which case the vector fields are integrable, or not, in which case the differential of the form defines a, possibly singular, symplectic form. In every case, we provide a fairly complete list of local models for such fields and construct their transversal unfoldings. Thus, the local bifurcations of vector fields of interest can be studied, among them being the integrable fields of the plane.
\end{abstract}


\section{Introduction}

The classification of dynamical systems preserving a geometric structure is excessively studied. Among many great references, we would like to point out \cite{Chow Li Wang,Belitskii Kopanskii, Belitskii Kopanskii 2,Banyaga Llave Wayne}, for classification results concerning vector fields preserving a volume/symplectic/contact structure, or possessing a symmetry.

In \cite{Luna, Martins Tari}, the authors deal with the very interesting problem of studying vector fields tangent to a foliation. More specifically, in \cite{Luna}, local models for vector fields tangent to the level sets of a Morse function were given, at least for an open and dense set of such vector fields, while, in \cite{Martins Tari}, the stability of pairs $(\omega,X)$, consisting of an integrable 1--form $\omega$ and a vector field $X$, tangent to the foliation defined by $\omega$, was studied.

Here we are interested in the problem of classifying vector fields which belong to the kernel of a 1--form and begin our study with vector fields of the plane. In section 2,  we present the equivalence relation which preserves not the 1--form of interest, but rather its kernel. This relation induces a partition of vector fields belonging in $\ker (a)$ into equivalence classes, where all the members of the same class share the same orbit structure. We prove that the curve of singularities of the vector field of interest determines the equivalence class of the field and use results obtained in \cite{Zhitomirskii 1}, to show how one can construct local models for each equivalence class.

In section 3, we present the case where the 1--form of interest is not closed. Due to their genericity, we study in detail the 1--forms of Darboux and Martinet, and we provide local models for all, simple, vector fields in their kernel. We also study the kernel of the  Liouville 1--form and construct local models for vector fields in its kernel, provided they are finitely determined. We also show how one can construct transversal unfoldings for these models, to study their bifurcations.

In section 4, we solve the problem in the case where the 1--form is closed. In this case, the vector fields in its kernel are integrable and we present, in detail, the classification and bifurcations of such fields when the first integral is regular at the origin, or possesses a degenerate singularity there. However, by using an example, we demonstrate how our results can be used to study bifurcations of integrable vector fields with an arbitrary first integral.

The last section contains some remarks on future work.

We emphasize that we work in the smooth (i.e. $C^{\infty}$) category. Although, for brevity, we do not state it repeatedly, the study that follows is local. To be more precise, we study germs at the origin of 1--forms and vector fields, of functions $f:(\mathbb{R}^2,0)\rightarrow \mathbb{R}$ and of diffeomorphisms $\phi: (\mathbb{R}^2,0)\rightarrow (\mathbb{R}^2,0)$. 

\section{Orbital (a,b)--conjugacy}

Let $\omega$ be a volume form of the plane. If $a$ is a fixed 1--form, the equation $X_a\ \lrcorner \ \omega = a$ defines a unique vector field $X_a$ corresponding to $a$, which generates the kernel of $a$, $\ker (a):=\{X\in \mathcal{X}(\mathbb{R}^2),a(X)=0\}$, i.e., all members of $\ker (a)$ are of the form $f\cdot X_a$, for some smooth function $f$.

Two differential 1--forms, $a,b$ are conformally equivalent, if a diffeomorphism $\phi$ exists, such that $a=k\cdot \phi^*b$, for some non--zero function $k$. In the case where $a=b$, $\phi$ is called a conformal symmetry of $a$, while we drop the term ``conformal", if $k\equiv 1$.  

There is an analogous notion for vector fields: two vector fields $X_1,X_2$ are called ``orbitally conjugate", if a non--zero function $k$ and a diffeomorphism $\phi$ exists, such that $X_1=k\cdot \phi^*X_2$.

The conformal equivalence of two 1--forms induces a specific orbital conjugacy between their corresponding vector fields.

\begin{definition}
Let $\omega$ be a volume form of the plane and $a,b$ two conformally equivalent 1--forms, i.e. a non--zero function $k$ and a diffeomorphism $\phi$ exists, such that $a=k\cdot \phi^*b$. Two vector fields $X_1,X_2$ will be called orbitally $(a,b)$--conjugate (orbitally $a$--conjugate, in case $a=b$), if a non--zero function $h$ exists, such that $X_1=h\cdot \phi^*X_2$. We shall denote this equivalence relation by $X_1\sim X_2$.
\end{definition}

We remark that the difference between the classical definition of orbital conjugacy of vector fields and the orbital $(a,b)$--conjugacy is that in the first case the diffeomorphism conjugating the vector fields can be chosen freely, while in the second case, the diffeomorphism should define a conformal equivalence between the forms $a,b$ as well. 

\begin{lemma}
Let $\omega$ be a volume form of the plane and $a,b$ two 1--forms. The two forms are conformally equivalent if, and only if, their corresponding vector fields are orbitally $(a,b)$--conjugate.
\end{lemma}
\begin{proof}
Let us suppose that a diffeomorphism $\phi$ and a non--zero function $k$ of the plane exist, such that $a=k\cdot \phi ^* b$. If $V$ is a vector field of the plane, we have:
\begin{equation*}
\begin{split}
a=k\cdot \phi^*b & \Leftrightarrow V\ \lrcorner \ a = V\ \lrcorner \ k\cdot \phi^*b \Leftrightarrow \\
\Leftrightarrow \omega (X_a, V) & = V \ \lrcorner \ k\cdot \phi^* \Big( \omega (X_b,\cdot)\Big) =\\
& = k\cdot \omega (X_b,\phi_*V) = \\
& = k\cdot \omega (\phi_*Y,\phi_*V)=\\
& = k\cdot \det (d\phi) \cdot \Big(\omega (\phi^*X_b,V)\Big),
\end{split}
\end{equation*}
where $Y=\phi^*X_b$. Hence $X_a=k\cdot \det(d\phi)\cdot \Big(\phi^*X_b\Big)$ and, since $\det(d\phi)\neq 0$, the vector fields $X_a,X_b$ are orbitally $(a,b)$--conjugate. 

On the other hand, if $X_a=\lambda \cdot \phi^*X_b$ and $V$ a vector field, we have:
\begin{equation*}
\begin{split}
a(V) & =\omega(X_a,V)=\omega(\lambda \cdot \phi^*X_b,V)=\omega(\lambda\cdot \phi^*X_b,\phi^*(\phi_*V))=\\
& =\lambda\cdot (\det d\phi)^{-1}\cdot\omega(X_b,\phi_*V)=k\cdot \phi^*b(V),
\end{split}
\end{equation*}
for $k=\lambda \cdot (\det (d\phi))^{-1}$.
\end{proof}

The conformal equivalence of forms induces an equivalence relation on the ring $\mathcal{E}$ of smooth function--germs at the origin.

\begin{definition}
Let $\omega$ be a volume form of the plane and $a,b$ two conformally equivalent 1--forms, i.e. a non--zero function $k$ and a diffeomorphism exists such that $a=k\cdot \phi^*b$. Two function--germs $f_1,f_2$ will be called $\mathcal{K}_{(a,b)}$--equivalent (or $\mathcal{K}_a$--equivalent, in case $a=b$), if a non--zero function $h$ exists, such that $f_1=h\cdot \phi^*f_2$. We shall denote this equivalence relation by $f_1\sim f_2$.
\end{definition}

In what follows, we shall frequently use the following:

\begin{proposition}
Let $\omega$ be a volume form of the plane, and $a,b$ two conformally equivalent 1--forms, with corresponding vector fields $X_a,X_b$. The function germs $f_1,f_2$ are $\mathcal{K}_{(a,b)}$--equivalent if, and only if, the vector fields $f_1\cdot X_a,f_2\cdot X_b$ are orbitally $(a,b)$--equivalent.
\end{proposition}
\begin{proof}
Since $a,b$ are conformally equivalent, a non--zero function $k$ and a local diffeomorphism exists such that $a=k\cdot \phi^*b$. It follows that $X_a=k\cdot (\det d\phi)\cdot \phi^*X_b$. Let $f_1=\lambda \cdot (f_2\circ \phi)$.

We have:
\begin{equation*}
\begin{split}
f_1\cdot X_a & =\lambda \cdot (f_2\circ \phi)\cdot k\cdot (\det d\phi)\cdot \phi^*X_b=\\
& = \lambda \cdot k \cdot (\det d\phi)\cdot \phi^*(f_2X_b).
\end{split}
\end{equation*}
The vector fields $f_1\cdot X_a,f_2\cdot X_b$ are therefore orbitally $(a,b)$--conjugate.

On the other hand, suppose that $f_1\cdot X_a,f_2 \cdot X_b$ are orbitally $(a,b)$--conjugate. A non--zero function $\lambda$ exists, such that:
\begin{equation*}
\begin{split}
f_1\cdot X_a & =\lambda \cdot \phi^*(f_2\cdot X_b)\Rightarrow \\
\Rightarrow f_1\cdot k\cdot (\det d\phi)\cdot \phi^*X_b &=\lambda \cdot (f_2\circ \phi)\cdot \phi^*X_b \Rightarrow \\
\Rightarrow f_1 & =\frac{\lambda}{k\cdot (\det d\phi)}\cdot f_2\circ \phi,
\end{split}
\end{equation*}
as claimed.
\end{proof}

\begin{example}
Let $b=dx$ and consider the non--zero function--germ $k(x,y)=2+xy$ and the local diffeomorphism $\phi(x,y)=(x+x^2,x+y)$. We then get that $a=k\cdot \phi^* b=(2+xy)\cdot (1+2x)\cdot dx$.

The vector fields corresponding to $a,b$, with respect to the standard volume form $dx\wedge dy$, are $X_a=-(2+xy)\cdot (1+2x)\frac{\partial}{\partial y}$ and $X_b=-\frac{\partial}{\partial y}$. Obviously, $X_a=k\cdot (\det d\phi)\cdot \phi^*X_b$. 

Moreover, for the the functions $f_1(x,y)=e^x(2x+y+x^2)$ and $f_2(x,y)=x+y$ the equation $f_1(x,y)=e^x\cdot f_2(\phi(x,y))$ holds and it is evident that $f_1\cdot X_a=e^x\cdot (2+xy)\cdot \phi^*(f_2\cdot X_b)$.
\end{example}

We are here interested in the case where $a=b$. In this case, the diffeomorphism $\phi$, above, is a conformal symmetry of $a$ and we wish to classify vector fields belonging in the kernel of $a$, that is, vector fields of the form $f_1\cdot X_a,f_2\cdot X_a,\ f_1,f_2\in \mathcal{E}$, up to orbital $a$--conjugacy. As we saw, to achieve that, we must classify the function germs $f_1,f_2$ up to $\mathcal{K}_a$--equivalence.

Let $f_0$ be a function and consider a curve belonging in $\mathcal{E}$, passing through $f_0$ for $t=0$, consisting of functions that are $\mathcal{K}_a$--equivalent to $f_0$. This curve is of the form $f_t=g_t\cdot (f_0 \circ \phi_t)$, where, $\forall t \in \mathbb{R}$, $g_t$ is a non--zero function, with $g_0(x,y)=1$ and $\phi_t$ is a conformal symmetry of $a$, satisfying $\phi_0=Id$. Differentiating $f_t$ with respect to $t$, we get:
\[
\frac{\partial f_t}{\partial t}=\frac{\partial g_t}{\partial t}\cdot (f_0\circ \phi_t)+g_t\cdot (k\cdot \mathcal{L}_{X_a \circ \phi_t}(f_0\circ \phi_t)),\ \ (\star)
\]    
where $X_a$  is the vector field defined as follows:
\[
\forall t,\  \frac{\partial \phi_t}{\partial t}=k\cdot (X_a \circ \phi_t),
\]
and $\mathcal{L}$ stands for the Lie derivative. The non--zero term $k\in \mathcal{E}$ is present in the equation above since $\phi_t$ preserves the foliation defined by $X_a$, not $X_a$ itself. 

Evaluating at $t=0$, relation $(\star)$ becomes:
\[
\frac{\partial f_t}{\partial t}\Big| _{t=0}=\frac{\partial g_t}{\partial 
t}\Big| _{t=0} \cdot f_0 + k\cdot \mathcal{L}_{X_a}(f_0).
\]
Since function $k$ is chosen freely, while $g_t$ depends on the curve $f_t$, the calculation above motivates the following:

\begin{definition}
Let $f\in \mathcal{E}$ and denote by $\mathcal{O}_f$ the set of all members of $\mathcal{E}$ which are $\mathcal{K}_a$--equivalent to $f$. The tangent space of $\mathcal{O}_f$ at $f$ is defined to be $\langle f,\mathcal{L}_{X_a}f\rangle$. The codimension of $f$ (and of the vector field $f\cdot X_a$), with respect to the $\mathcal{K}_a$--equivalence relation (orbital $a$--conjugacy), is defined to be $codim(f)=codim(f\cdot X_a):=\dim \big(\mathcal{E} / \langle f,\mathcal{L}_{X_a}f\rangle \big)$.  
\end{definition}

Recall that, in the definitions above, the conjugating diffeomorphism $\phi$ is a conformal symmetry of the 1--form $a$, thus it preserves not the vector field $X_a$ itself, but rather the, possibly singular, foliation defined by it. Since we do not distinguish between a function and its non--zero multiples, we actually wish to classify principal ideals of the ring $\mathcal{E}$ (that is, ideals of the form $\langle f\rangle,\ f\in \mathcal{E}$) up to diffeomorphisms preserving the foliation defined by $X_a$.

In \cite{Zhitomirskii 1}, the problem of classifying plane curves:
\[
\{f=0\}:=\{(x,y)\in \mathbb{R}^2,f(x,y)=0\},
\]
up to diffeomorphisms preserving the foliation given by some fixed vector field $X$, was solved. Assuming that the plane curves satisfy the property of zeros (that is, two functions, vanishing at exactly the same points, are non--zero multiples of each other) the classification given there coincides with the classification of principal ideals of $\mathcal{E}$. Identifying the ideal $\langle f\rangle$ with the curve $\{f=0\}$ and using the results of \cite{Zhitomirskii 1}, we are able to provide our first classification results.

A 1--form $a$ (or a vector field $X$) is regular at the origin if $a(0,0)\neq 0$ ($X(0,0)\neq 0$), otherwise we call it ``singular". A curve $\{f=0\}$, passing through the origin, is regular if $d_0f\neq 0$ and singular otherwise. Two curves $\{f=0\},\{g=0\}$ have order of tangency $k\in \mathbb{N}$ if $j^kf(0)=j^kg(0)$, where $j^kf$ denotes the k--jet of the function $f$.

\begin{theorem}[on regular 1--forms] \label{Theorem regular}\\
Let $a$ be a regular 1--form of the plane, with corresponding vector field $X_a$, and $X\in \ker (a),\ X=f\cdot X_a,\ f\in \mathcal{E}$.
\begin{enumerate}
\item {If $X$ is regular at the origin, $X\sim X_{a}$. The equivalence class of these vector fields is of codimension $0$ in $\ker(a)$.}
\item {If the zeroes of $X$ form a regular curve passing through the origin, which curve has contact of order $k\in \mathbb{N}$ with the $x$--axis, $X\sim (y-x^{k+1})\cdot X_a$. This equivalence class is of codimension $k$ in $\ker(a)$.}
\item {If the curve of the singularities of $X$ is singular at the origin but simple, $X$ is orbitally $a$--conjugate with one of the following vector fields:
\[
X_1^k=(xy-x^k)\cdot X_a,\ X_k^{k+1}=(x^2\pm y^{k+1})\cdot X_a,\ X_2^4=(y^2+x^3)\cdot X_a.\]
Here, $k\geq 2$ and the codimension of their equivalence classes is $k,k+1,4$ respectively.}
\end{enumerate}
\end{theorem} 
\begin{proof}
\begin{enumerate}\item[]
\item {Since $X=f\cdot X_a$ is regular, function $f$ does not vanish at the origin, in which case $f\sim 1$ (just take $\phi=Id,\ g=f^{-1}$ in the definition of $\mathcal{K}_a$--equivalence). A simple calculation confirms that $codim(f)=0$.}
\item {Take local coordinates in which $X_a=\frac{\partial}{\partial x}$. Let us suppose that the, regular, curve of singularities of $X$ is transverse to the $y$--axis, i.e. it is of the form $(x,g(x))$. Since it has order of tangency $k$ with the $x$--axis, it can be written as $(x,x^{k+1}h(x))$, where $h(0)\neq 0$. This curve corresponds to the function $\tilde{f}(x,y)=y-x^{k+1}h(x)$, while the local diffeomorphism $\phi(x,y)=(x\cdot h^{\frac{1}{k+1}}(x),y)$ preserves the foliation defined by $X_a$ and satisfies equation $f\circ \phi=\tilde{f}$, for $f(x,y)=y-x^{k+1}$.

If the, regular, curve is transverse to the $x$--axis, it is of the form $(g(y),y)$ and it corresponds to the function $k(x,y)=x-g(y)$. The local diffeomorphism $(x,y)\mapsto (x+g(y),y)$ preserves the foliation defined by $X_a$ and satisfies equation $k\circ \phi =x$, which function is $\mathcal{K}_a$--equivalent to the function $y-x$.}
\item{This follows from \cite[Theorem 5.2]{Zhitomirskii 1}: A curve $\{f=0\}$ defined on a plane which is equipped with a foliation is simple if the pair consisting of the curve and the foliation is simple, with respect to the action of diffeomorphisms preserving this foliation. In this case, the curve is equivalent to one of the following curves:
\[
\{xy+x^k=0\},\ \{x^2\pm y^{k+1}=0\},\ \{y^2+x^3=0\},
\]}
hence the conclusion.
\end{enumerate}
\end{proof}

\begin{remark}
In the third statement of \ref{Theorem regular}, the ``$\pm$" sign should be replaced with $``+"$, in case $k$ is even.
\end{remark}

The analogous statement for singular 1--forms follows.

\begin{theorem}[on singular 1--forms] \label{Theorem singular}\\
Let $a$ be a singular 1--form of the plane and $X\in \ker (a),\ X=f\cdot X_a,\ f\in \mathcal{E}$.
\begin{enumerate}
\item {If $f(0,0)\neq 0$, $X\sim X_{a}$. Their equivalence class forms a set of codimension $0$ in $\ker(a)$.}

\item {If $f(0,0)=0$ and the curve $\{f=0\}$ is of finite codimension $k$, $f\sim j^kf$ and  therefore $X\sim j^kf\cdot X_a$.}
\end{enumerate}
\end{theorem}
\begin{proof}
\begin{enumerate}\item[]
\item {Since $f(0,0)\neq 0$, as before, $f\sim 1$.}
\item {This follows from \cite[Theorem A]{Zhitomirskii 1}. According to this theorem, and using our notation, $f\sim j^kf$ and thus $X\sim j^kf\cdot X_a$.}
\end{enumerate}
\end{proof}

In the next sections, using the two theorems above, we shall construct local models for vector fields belonging in the kernel of various  1--forms. 

\section{The case of non--closed forms}

\subsection{The 1--forms of Darboux and Martinet }

In the neighbourhood of every point of the plane, except possibly from points forming a zero--dimensional set, a 1--form is equivalent to either the Darboux model $a_D=(1+x)dy$, or to the Martinet model $a_M=(1\pm x^2)dy$. These two models are the only stable models for 1--forms of the plane, see \cite{Zhitomirskii 2}.

We note that both forms are regular, while $da_D$ equals the standard symplectic form of the plane $dx\wedge dy$ and $da_M$ equals $\pm 2xdx\wedge dy$, which does not define a symplectic structure everywhere, but is nevertheless related to Hamiltonian systems with constraints, see \cite{Kourliouros 1}. 

The vector fields, corresponding to these two 1--forms are:
\[
X_{a_D}=-(1+x)\frac{\partial}{\partial x},\ X_{a_M}=-(1\pm x^2)\frac{\partial}{\partial x},
\]
respectively.

The next theorem classifies vector fields belonging in $\ker(a)$, where $a$ is either the Darboux or the Martinet model. We omit the proof, since it is just an application of \ref{Theorem regular}. 

\begin{theorem}\label{Theorem Darboux Martinet}
Let $X\in \ker (a)$, where $a=a_D$ or $a=a_M$.
\begin{enumerate}
\item {If $X$ is regular at the origin, $X\sim X_{a}$. Their equivalence class forms a set of codimension $0$ in $\ker(a)$.}
\item {If the zeroes of $X$ form a regular curve passing through the origin, which curve has contact of order $k\in \mathbb{N}$ with the $x$--axis, $X\sim f\cdot X_a$, where $f(x,y)=y-x^{k+1}$. This equivalence class is of codimension $k$ in $\ker(a)$.}
\item {If the curve of the singularities of $X$ is singular at the origin but simple, $X$ is orbitally $a$--conjugate with one of the following vector fields:
\[
X_1^k=(xy-x^k)\cdot X_a,\ X_k^{k+1}=(x^2\pm y^{k+1})\cdot X_a,\ X_2^4=(y^2+x^3)\cdot X_a.\]
Here, $k\geq 2$ and the codimension of their equivalence classes is $k,k+1,4$ respectively. If $k$ is even, the sign $``\pm"$ should be changed to $``+"$.}
\end{enumerate}
\end{theorem}

We now wish to study the local bifurcations of the vector fields belonging in $\ker(a_D)$ and $\ker(a_M)$. To achieve this, we make use of the fact that the mapping $f\mapsto f\cdot X$ is a linear isomorphism. Thus, the construction of transversal unfoldings for vector fields belonging in the kernel of a 1--form can be achieved, by studying transversal unfoldings of functions.

Since the equivalence class of $X_a$ is of codimension 0, this vector field undergoes no local bifurcations. Generic bifurcations of codimension $k$ are given in the following:

\begin{theorem}\label{Theorem bif Darboux Martinet}
Let $a$ be either the Darboux or Martinet 1--form. The set of vector fields belonging in $\ker(a)$, having a curve of fixed points regular at the origin and with order of contact $k$ with the $x$--axis, has codimension k in $\ker (a)$. All these vector fields are orbitally $a$--conjugate to the $(y-x^{k+1}) \cdot X_a$ vector field. The family of vector fields:
\[
(y+\sum_{n=0}^{k-1}c_{n}x^n-x^{k+1})\cdot X_{a}
\]
intersects, for $c_i=0,\ i=0,..,k-1$, this equivalence class transversely.
\end{theorem}  
\begin{proof}
We prove only the case where $a=a_D$. 

As we saw, in \ref{Theorem Darboux Martinet}, if the curve of fixed points of the vector field $X\in \ker(a)$ is regular and has order of contact $k$ with the $x$--axis, this field is orbitally $a$--equivalent with $(y-x^{k+1})\cdot X_a$. The tangent space of the equivalence class $\mathcal{O}_f$ at $f$, for $f(x,y)=y-x^{k+1}$, equals $\langle f,\mathcal{L}_{X_a}f\rangle =\langle x^k,y\rangle$, hence $codim ((y-x^k)\cdot X_a)=k$. The functions $1,..,x^{k-1}$ do not belong to $\langle x,y\rangle$ or to $\mathcal{O}_f$, hence the family of functions $y+\sum_{n=0}^{k-1}c_{n}x^n-x^{k+1}$ is transverse to $\mathcal{O}_f$ at $f$. The linear isomorphism $h\mapsto h\cdot X_a$ maps this family to the family curve of vector fields presented above. 
\end{proof}

\begin{figure}
\begin{center}
\includegraphics[scale=0.38]{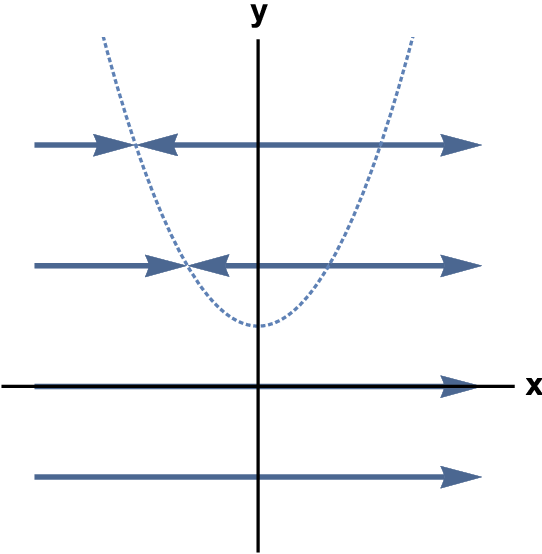}
\includegraphics[scale=0.38]{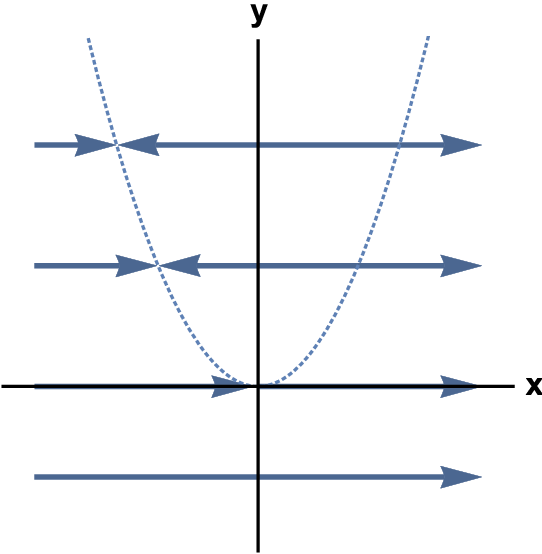}
\includegraphics[scale=0.38]{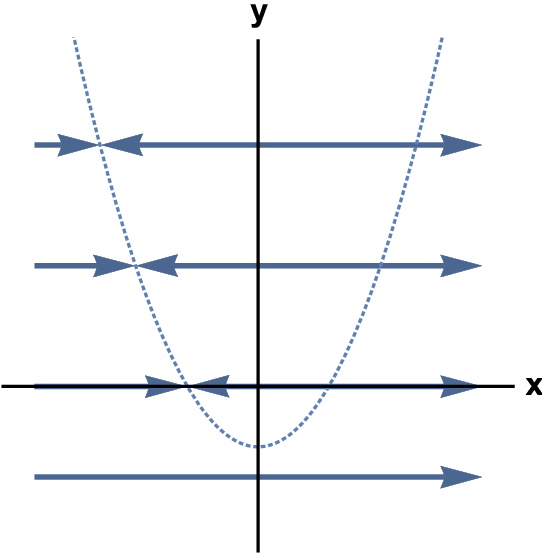}
\caption{Local bifurcation of codimension 1, for vector fields belonging in the kernel of the Darboux 1--form. The phase portrait of the vector field $(c+y-x^2)\cdot X_{a_D}$ is shown, for $c<0$ (left), $c=0$ (center) and $c>0$ (right). The dotted curve represents the curve of equilibria. Only a sufficient small neighbourhood of the origin is shown.}
\label{figure Darboux}
\end{center}
\end{figure}

In figure (\ref{figure Darboux}), we present the local bifurcation of codimension 1 family $(c+y-x^2)\cdot X_D$ undergoes. The case with the $(c+y-x^2)\cdot X_M$ is completely analogous, in a small neighbourhood of $(0,0)$, the difference being in the presence of an additional lines of singularities, since $X_D$ possesses a line of singularities at $x=-1$, while $X_M$ possesses no lines of singularities, in the case of positive sign and two lines of singularities, in the case of the negative sign.

\subsection{The 1--form of Liouville}

The Liouville form $a_L=xdy$ is of special importance, since its the primitive form of the standard volume form of the plane. This form, along with its corresponding vector field $X_{a_L}=-x\frac{\partial}{\partial x}$, are singular at the origin. Using the previous results, we can state the following:

\begin{theorem}\label{Theorem aL form}
Let $X=f\cdot X_{a_L}$.
\begin{enumerate}

\item {If $f(0,0)\neq 0$, $X\sim X_{a_L}$. Their equivalence class forms a set of codimension $0$ in $\ker(a_L)$.}

\item {If $f$ is regular at the origin and the curve $\{f=0\}$ is transversal to both coordinate axes there,   $X\sim (x+y)\cdot X_{a_L}$. This equivalence class is the unique class of codimension 1 in $\ker(a_L)$.}

\item {If $f(0,0)=0$ and of codimension $k$, $X\sim j^kf\cdot X_{a_L}$.}
\end{enumerate}
\end{theorem} 
\begin{proof}
\begin{enumerate}\item[]
\item {As before, in this case $f\sim 1$.}

\item {Let us suppose that the curve is transversal to the $y$--axis. It is then of the form $\{y=g(x)\},\ g(0)=0$. The transversality of the curve to $x$--axis imposes the condition $g'(0)\neq 0$; thus $g(x)=x\cdot k(x),k(0)\neq 0$. We now calculate that:
\begin{equation*}
\begin{split}
\ & \big \langle y-x\cdot k(x),\ \mathcal{L}_{X_L}\big (y-x\cdot k(x)\big)\big \rangle = \\
&= \big \langle y-x\cdot k(x),\ x\cdot k(x)+x^2\cdot k'(x)\big \rangle =\\
&= \big \langle y-x\cdot k(x),\ y+x^2\cdot k'(x)\big \rangle=\\
&= \big \langle -x\cdot k(x)-x^2\cdot k'(x),\ y+x^2\cdot k'(x)\big \rangle=\\
&= \big \langle x\cdot \big (-k(x)-x\cdot k'(x)\big ),\ y+x^2\cdot k'(x)\big \rangle=\\
&= \big \langle x,\ y+x^2\cdot k'(x)\big \rangle=\\
&= \big \langle x,y \big \rangle.
\end{split}
\end{equation*}
The codimension, in this case, is therefore equal to 1. According to \cite[Theorem 4.1]{Zhitomirskii 1}, all functions which are regular at the origin and of codimension equal to 1 are $\mathcal{K}_{a_L}$--equivalent to the function $x+y$. Hence the conclusion.}

\item {This follows from \ref{Theorem singular}.}
\end{enumerate}
\end{proof}

We now describe the local bifurcation of codimension 1, for vector fields belonging in $\ker (a_L)$, omitting the proof. 

\begin{theorem}\label{Theorem bif Liouville}
Let $a_L=xdy$ be the form of Liouville. The set of vector fields belonging in $\ker(a_L)$, having a curve of fixed points regular at the origin and transversal to both coordinate axes, has codimension 1 in $\ker (a_L)$. All these vector fields are orbitally $a_L$--conjugate to the $(x+y) \cdot X_{a_L}$ vector field. The family of vector fields $(c+x+y)\cdot X_{a_L}$ intersects, for $c=0$, this equivalence class transversely.
\end{theorem}  

In figure \ref{figure Liouville}, the codimension 1 bifurcation is depicted.

\begin{figure}
\begin{center}
\includegraphics[scale=0.38]{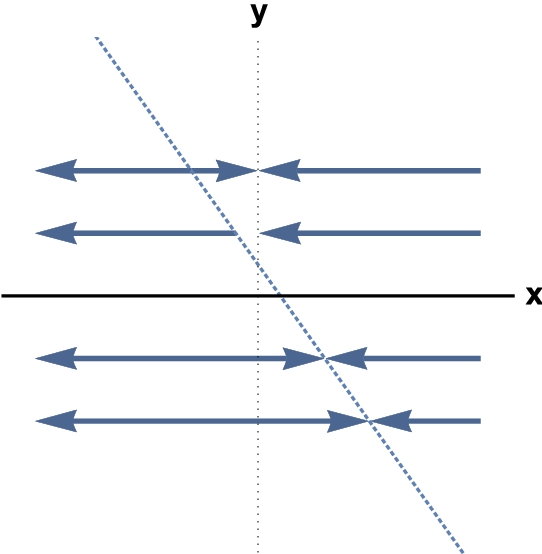}
\includegraphics[scale=0.38]{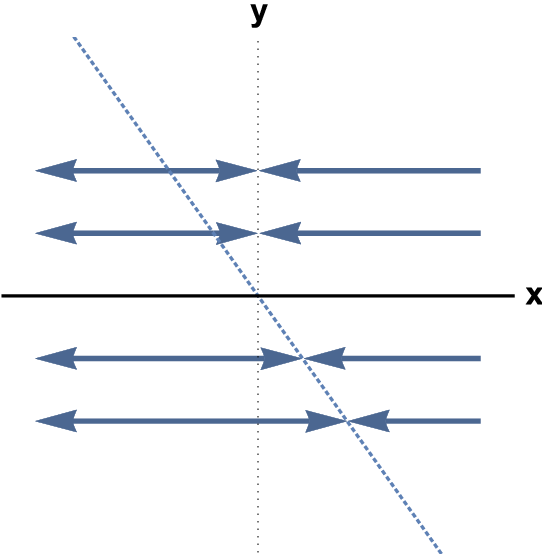}
\includegraphics[scale=0.38]{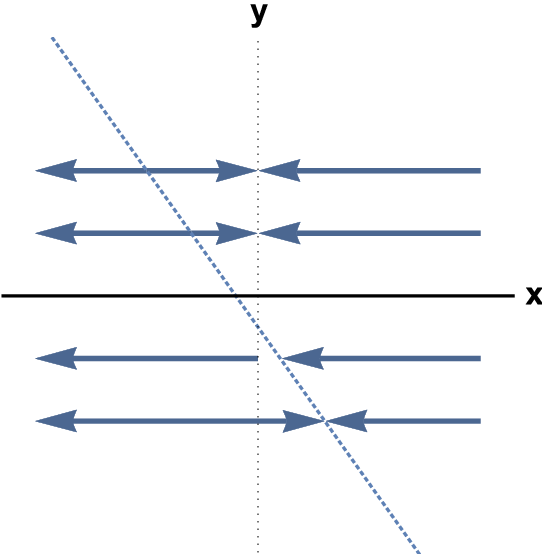}
\caption{Local bifurcation of codimension 1, for vector fields belonging in the kernel of the Liouville 1--form. The phase portrait of the vector field $(c+x+y)\cdot X_{a_L}$ is shown, for $c<0$ (left), $c=0$ (center) and $c>0$ (right). The dotted curves represents the curve of equilibria.}
\label{figure Liouville}
\end{center}
\end{figure}

\begin{remark}
Bifurcations of higher codimensions can be studied using the same approach, although one need to be more careful, since equivalence classes of codimension $k\geq 2$ need not be unique. As is easily verifiable, for example, both vector fields $(x-y^2)\cdot X_{a_L},\ (y-x^2)\cdot X_{a_L}$ are of codimension 2; however they are not orbitally $a_L$--conjugate.
\end{remark}


\section{The case of closed 1--forms}

The case where the 1--form of interest is closed is of special importance, since the kernel of the form consists of integrable vector fields. In other words, in this section we aim to give a complete description of the local models and bifurcations, of arbitrary codimension, of planar integrable vector fields.

We present in detail the case where the first integral is either regular at the origin (and therefore, local coordinates exist, in which the integral takes the simple form $I(x,y)=y$) and to the case where the integral is singular at the origin, but the singularity is of Morse type (it is thus equal, in local coordinates, to $x^2\pm y^2$. However, following the same lines, vector fields having an arbitrary integral can be studied, as we demonstrate with an example.
\subsection{The kernel of $a_1=dy$}

All 1--forms, that are regular at $0\in \mathbb{R}^2$, are equivalent to the form $a_1=dy$. The vector fields belonging in $\ker(a_1)$ are of the form $f\cdot X_{a_1}$, where $X_{a_1}=-\frac{\partial}{\partial x}$. These vector fields are, of course, integrable, the integral being the function $y$.

Since the 1--form $a_1$, and its corresponding vector field, are regular at the origin, one can state, using \ref{Theorem regular}, the following:

\begin{theorem}\label{Theorem a1 form}
Let $X\in \ker (a_1)$.
\begin{enumerate}
\item {If $X$ is regular, $X\sim X_{a_1}$. Their equivalence class forms a set of codimension $0$ in $\ker(a_M)$.}
\item {If the zeroes of $X$ form a regular curve passing through the origin, which curve has contact of order $k$ with the $x$--axis, $X\sim (y-x^{k+1})\cdot X_{a_1}$. This equivalence class is of codimension $k$.}
\item {If the curve of the singularities of $X$ is singular at the origin but simple, $X$ is orbitally $a_{1}$--conjugate with one of the following vector fields:
\[
X_1^k=(xy-x^k)\cdot X_{a_1},\ X_k^{k+1}=(x^2\pm y^{k+1})\cdot X_{a_1},\ X_2^4=(y^2+x^3)\cdot X_{a_1}.
\]
Here, $k\geq 2$, while their equivalence classes are of codimension $k,k+1,4$ respectively. If $k$ is even, the $``\pm"$ sign should chang to $``+"$.}
\end{enumerate}
\end{theorem} 

\begin{remark}
One should compare our results with \cite[Theorem 3.1]{Oliveira Tari}. The function $y+xy-x^3$, presented there, is $\mathcal{K}_{a_1}$--equivalent to the function $y-x^3$, presented above. Our results are also connected with the classification of pairs of plane hamiltonian vector fields, see \cite[Theorem A]{Oliveira}.
\end{remark}

We now present local bifurcations of codimension k, for vector fields with function $y$ as a first integral.

\begin{theorem}
Let $X=f\cdot X_{a_1}$ be a vector field of the plane, having the function $I(x,y)=y$ as a first integral. If the curve of its singularities $\{f=0\}$ is regular at the origin and of order of contact $k$ with the $x$--axis, the vector field is orbitally conjugate, via a diffeomorphism preserving the integral $I$, to the vector field $-(y-x^{k+1})\cdot X_{a_1}$. This equivalence class is of codimension $k$ in $\ker(a)$. The family of vector fields:
\[
(y+\sum_{n=0}^{k-1}c_{n}x^n-x^{k+1})\cdot X_{a_1}
\] intersects, for $c_i=0,\ i=0,..,k-1$, this equivalence class transversely.
\end{theorem}

We omit the proof of this theorem, since it follows the lines of the proof of \ref{Theorem bif Darboux Martinet}. The phase portrait of codimension 1 bifurcation is almost identical with the one presented in figure \ref{figure Darboux}, the only difference being the extra line of fixed points, in the Darboux case, which is away from the origin. 

\subsection{The kernel of $a_2^+=xdx+ydy$}

The form $a_2^+=xdx+ydy$ is singular at the origin and so is its corresponding vector field $X_{a_2^+}=-y\frac{\partial}{\partial x}+x\frac{\partial}{\partial y}$. All the vector fields belonging in $\ker(a_2^+)$ share the first integral $x^2+y^2$ and have, necessarily, one zero at the origin. Their classification is given in the following:

\begin{theorem}\label{Theorem a2+ form}
Let $X=f\cdot X_{a_2^+}$.
\begin{enumerate}

\item {If $X$ possesses only one zero, it is orbitally $a_{2}^+$--conjugate to $X_{a_2^+}$. Their equivalence class forms a set of codimension $0$ in $\ker(a_2^+)$.}

\item {If $f\cdot X_{a_2^+}$ is such that the curve $\{f=0\}$ is regular at the origin, the vector field is orbitally $a_2^+$--conjugate to $-xy\frac{\partial}{\partial x}+x^2\frac{\partial}{\partial y}$.  This equivalence class is of codimension $1$.}

\item {If the curve $\{f=0\}$ of the singularities of $X$ is of finite codimension $k$, $f\sim j^kf$ and, therefore, $X\sim j^kf\cdot X_{a_2^+}$.}

\end{enumerate}
\end{theorem} 
\begin{proof}
\begin{enumerate}\item[]

\item {In this case $f(0,0)\neq 0$ and thus $f\sim 1$.}

\item {Let $\{f=0\}$ be non--singular at the origin. The relation:
\[
codim (\mathcal{E}/\langle f,\mathcal{L}_{X_{a_2^+}} f\rangle)\geq 2
\]
holds if, and only if, the restriction of $\mathcal{L}_{X_{a_2^+}} f$ on $\{f=0\}$ has zero 1--jet. Since:
\begin{equation*}
\begin{split}
j^1_{(x_0,y_0)}(\mathcal{L}_{X_{a_2^+}}f) & =-f_x(x_0,y_0)\cdot y+f_y(x_0,y_0)\cdot x= \\
& =(-y,x)\cdot (f_x(x_0,y_0),f_y(x_0,y_0)),
\end{split}
\end{equation*}  
we conclude that, at every point of curve $\{f=0\}$, the vector $X_{a_2^+}$ should be tangent to this curve; such a curve, however, cannot be regular.  

Therefore, the codimension of $f$ is 1. According to \cite[Theorem 4.1]{Zhitomirskii 1}, all regular curves of codimension 1 are equivalent, via a diffeomorphism preserving the foliation defined by $X_{a_2^+}$; hence the conclusion.}

\item {This follows from \ref{Theorem singular}.}
\end{enumerate}
\end{proof}

The bifurcation of codimension 1 is described in the following:

\begin{theorem}
Let $X=f\cdot X_{a_2^+}$ be a vector field of the plane, having the function $I(x,y)=x^2+y^2$ as a first integral. If the curve of its singularities $\{f=0\}$ is regular at the origin, the vector field is orbitally conjugate, via a diffeomorphism preserving the integral $I$, to the vector field $x\cdot X_{a_2^+}$. This equivalence class is of codimension $1$ in $\ker(a)$. The family of vector fields:
\[
(c+x)\cdot X_{a_2^+}
\]
intersects, for $c=0$, this equivalence class transversely.
\end{theorem}

This bifurcation is presented in figure (\ref{figure a2+}).

\begin{figure}
\begin{center}
\includegraphics[scale=0.38]{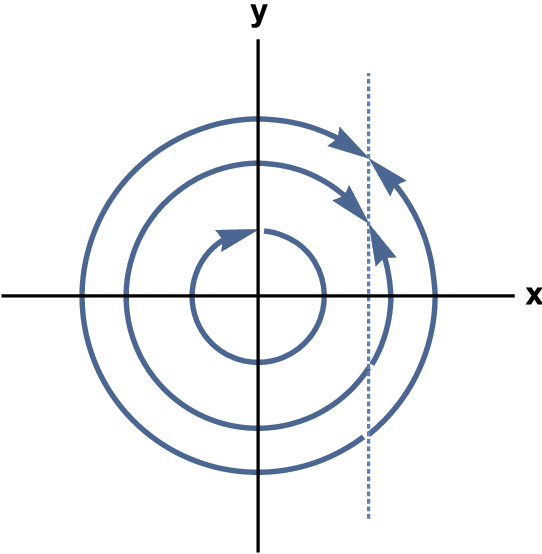}
\includegraphics[scale=0.38]{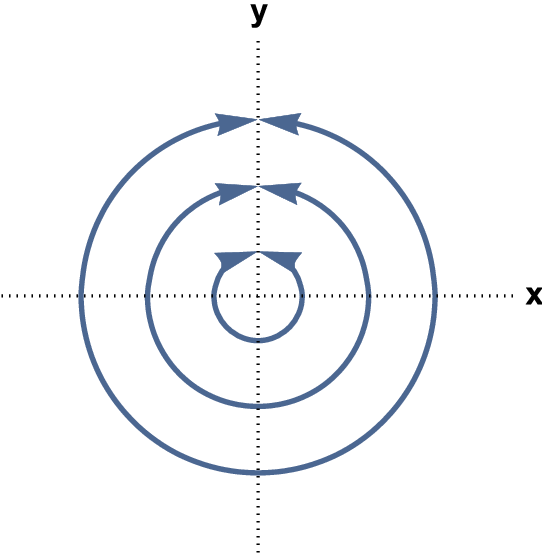}
\includegraphics[scale=0.38]{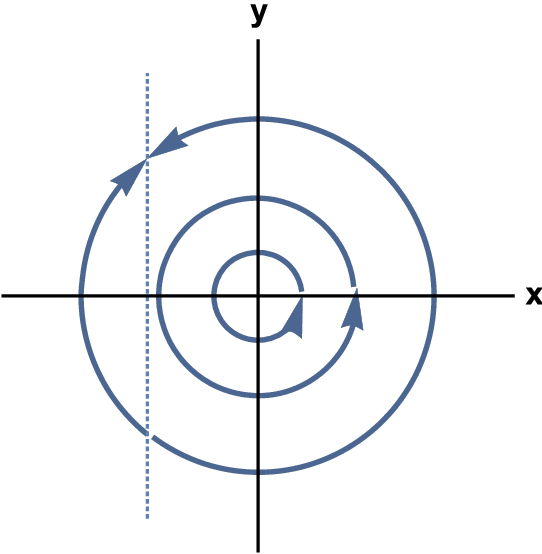}
\caption{Local bifurcation of codimension 1, for vector fields belonging in the kernel of the $a_2^+$--form. The phase portrait of the vector field $(c+x+y)\cdot X_{a_2^+}$ is shown, for $c<0$ (left), $c=0$ (center) and $c>0$ (right). The dotted curves represents the curve of equilibria.}
\label{figure a2+}
\end{center}
\end{figure}

\subsection{The kernel of $a_2^-=xdx-ydy$}

The form $a_2^-=xdx-ydy$ is singular at the origin and so is its corresponding vector field $X_{a_2^-}=y\frac{\partial}{\partial x}+x\frac{\partial}{\partial y}$. All the vector fields belonging in $\ker(a_2^-)$ have the function $x^2-y^2$ as a first integral. Their orbital $a_2^-$--conjugacy classes are given in the following:

\begin{theorem}\label{Theorem a2- form}
Let $X=f\cdot X_{a_2^-}$.
\begin{enumerate}

\item {If $X$ possesses only one zero, it is orbitally $a_{2}^-$--conjugate to $X_{a_2^-}$. Their equivalence class forms a set of codimension $0$ in $\ker(a_2^-)$.}

\item {If $f\cdot X_{a_2^-}$ is such that the curve $\{f=0\}$ is regular at the origin and of codimension $1$, the vector field is orbitally $a_2^-$--conjugate to $x\cdot X_{a_2^-}$.}

\item {If the curve $\{f=0\}$ of the singularities of $X$ is of finite codimension $k$, $X\sim j^kf \cdot X_{a_2^-}$.}

\end{enumerate}
\end{theorem} 
\begin{proof}
\begin{enumerate}\item[]

\item {In this case $f(0,0)\neq 0$ and thus $f\sim 1$.}

\item {There exist local coordinates in which the vector field takes the form $Y=x\frac{\partial}{\partial x}-y\frac{\partial}{\partial y}$. In those coordinates, the curve $\{f=0\}$ is equivalent, via a diffeomorphism preserving the foliation defined by $Y$, to the function $x+y$ (see \cite[Theorem 4.1]{Zhitomirskii 1}). Define $\phi(x,y)=(x+y,x-y)$. This mapping is a local diffeomorphism which satisfies the equation $X_{a_2^-}=\phi^*Y$. Since $\phi^*(x+y)=2x$, we arrive at our conclusion.}

\item {This follows from \ref{Theorem singular}. According to this theorem, and using our notation, $f\sim j^kf$ and thus $X\sim j^kX_{a_2^-}$.}
\end{enumerate}
\end{proof}

\begin{theorem}
Let $X=f\cdot X_{a_2^-}$ be a vector field of the plane, having the function $I(x,y)=x^2-y^2$ as a first integral. If the curve of its singularities $\{f=0\}$ is regular at the origin and of order of contact 1 with either one of the lines $\{x=y\},\ \{x=-y\}$, the vector field is orbitally conjugate, via a diffeomorphism preserving the integral $I$, to the vector field $x\cdot X_{a_2^-}$. This equivalence class is of codimension $1$ in $\ker(a)$. The family of vector fields:
\[
(c+x)\cdot X_{a_2^-}
\]
intersects, for $c=0$, this equivalence class transversely.
\end{theorem}
\begin{proof}
As in the proof of the second statement of \ref{Theorem a2- form}, take local coordinates in which $X_{a_2^-}$ takes the form $Y=x\frac{\partial}{\partial x}-y\frac{\partial}{\partial y}$. It is easy to show that, in these coordinates, a curve which is regular at the origin has codimension equal to 1 if, and only if, it has contact of order 1 with either one of the coordinate axes, while the diffeomorphism $\phi$ transfomrs these two axes to the lines mentioned in the statement. We leave the rest of the proof to the reader. 
\end{proof}

We depict this bifurcation in figure (\ref{figure a2-}).

\begin{figure}
\begin{center}
\includegraphics[scale=0.38]{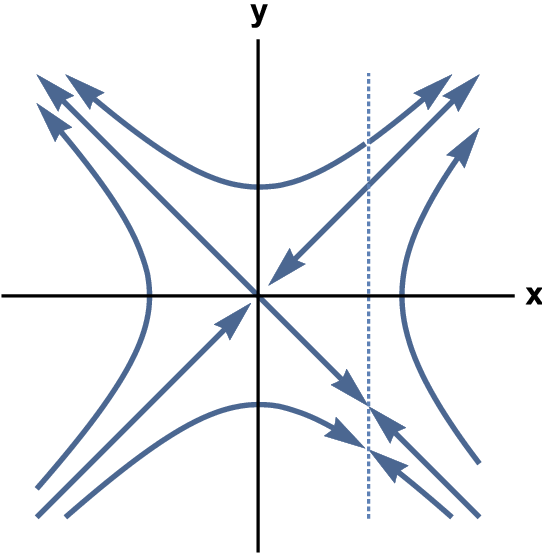}
\includegraphics[scale=0.38]{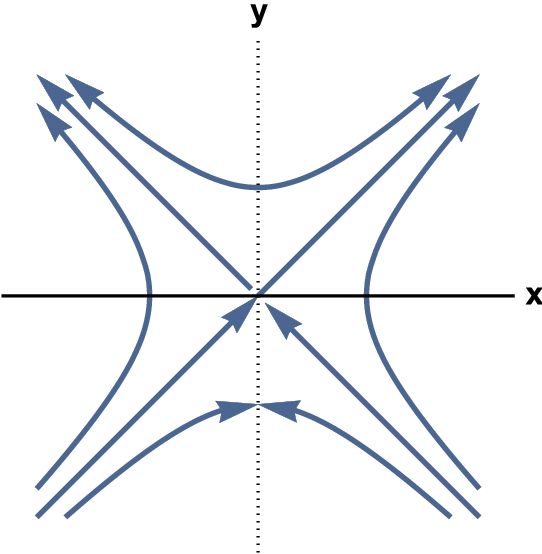}
\includegraphics[scale=0.38]{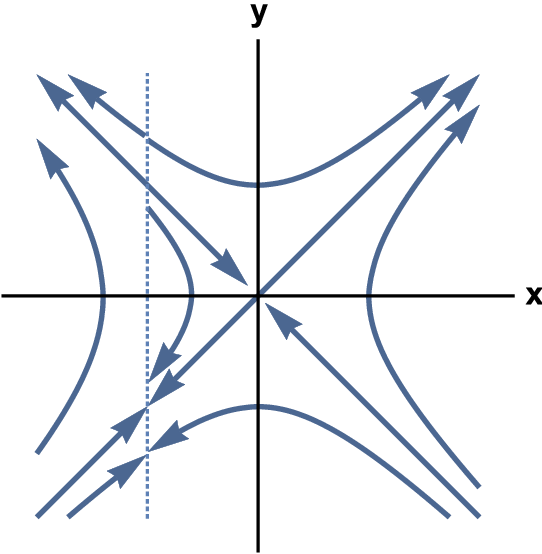}
\caption{Local bifurcation of codimension 1, for vector fields belonging in the kernel of the 1--form $a_2^-$. The phase portrait of the vector field $(c+x+y)\cdot X_{a_2^-}$ is shown, for $c<0$ (left), $c=0$ (center) and $c>0$ (right). The dotted curve represents the curve of equilibria.}
\label{figure a2-}
\end{center}
\end{figure}

\subsection{The kernel of $a_g=dg$}

In the same spirit, one can study the vector fields belonging in the kernel of every 1--form $a_k$, where $a_k=dg$, $g$ being a function of the plane. Let us give an example.

\begin{example}
Let $g(x,y)=\frac{1}{3}x^3-\frac{1}{2}x^2+\frac{1}{2}y^2$. We wish to study bifurcations of vector fields, having $g$ as a first integral. 

Here, $dg=(x^2-x)dx+ydy$, with associated vector field $X_g=-y\frac{\partial}{\partial x}+(x^2-x)\frac{\partial}{\partial y}$. This vector field is singular at the origin and all vector fields, having $g$ as a first integral, are of the form $f\cdot X_g$. Invoking \ref{Theorem singular}, one can deduce the following:
\begin{enumerate}
\item {A vector field $f\cdot X_g$, with $f(0,0)\neq 0$, is orbitally conjugate to $X_g$, via a diffeomorphism preserving the first integral.

For example, the vector field $-e^{x^2}y\frac{\partial}{\partial x}+(x^2-x)e^{x^2}\frac{\partial}{\partial y}$ has the function $g$ as a first integral and it is orbitally conjugate to $X_g$, via a diffeomorphism preserving $g$.}

\item {The vector field $Y=(x+y^2\cos y)\cdot X_g$ is orbitally equivalent, via a diffeomorphism preserving the integral $g$, to the vector field $(x+y^2)\cdot X_g$. This is because $codim (x+y^2\cos y)=2$ and $j^2(x+y^2\cos y)=x+y^2$.}

\item {Let $f(x,y)=ax+by+h.o.t$. The function $f$ is of codimension 1 if, and only if, one of the following holds:
\begin{itemize}
\item {$a\neq 0,b=0$. In this case the vector field $f\cdot X_g$ is orbitally $dg$--equivalent to $ax\cdot X_g$ and a transversal unfolding of it is $X_1=(c+ax)\cdot X_g$.}

\item {$a=0,b\neq 0$. In this case the vector field $f\cdot X_g$ is orbitally $dg$--equivalent to $by\cdot X_g$ and a transversal unfolding of it is $X_1=(c+by)\cdot X_g$.}

\item {$a,b\neq 0,a^2\neq b^2$. In this case the vector field $f\cdot X_g$ is orbitally $dg$--equivalent to $(ax+by)\cdot X_g$ and a transversal unfolding of it is $X_1=(c+ax+by)\cdot X_g$.}
\end{itemize}
}
\end{enumerate}   
\end{example}

\section{Conclusions}
We have presented local models for singularities of vector fields which belong to the kernel of a 1--form of the plane. This led naturally to a study of local bifurcations of integrable vector fields, in dimension 2.

As we mentioned before, in \cite{Luna, Martins Tari}, local models were presented, at least in the generic case, for integrable vector fields in any dimension. It would be very interesting to extend our results to dimensions greater than 2. This would not only provide a different proof for the results contained in \cite{Luna, Martins Tari}, but also present models for the non--generic cases as well.

We hope to be able to comment more on this, in a future publication. 


\begin{thebibliography}{99}

\bibitem{Luna}
Lara Luna G A, ``Estudo Local dos Campos Vetoriais com uma Integral Primeira de Morse", PhD Thesis, IMPA, 1978.

\bibitem{Zhitomirskii 2}
Zhitomirskii M, ``Typical Singularities of Differential 1--Forms and Pfaffian Equations", Translations of Math.Monographs, AMS, 1992.

\bibitem{Chow Li Wang}
S N Chow, C Li, D Wang, ``Normal Forms and  Bifurcations of Planar Vector Fields", Cambridge University Texts, 1994.

\bibitem{Banyaga Llave Wayne}
A Banyaga, R de la Llave, C E Wayne, ``Cohomology equations near hyperbolic points and geometric versions of sternberg linearization theorem", J.Geometric Analysis, 6, 613--649, 1996. 

\bibitem{Belitskii Kopanskii 2}
Belitskii G R, Kopanskii A Y, ``Sternberg theorem for equivariant Hamiltonian vector fields", Nonlinear Analysis, 47, 4491-4499, 2001.

\bibitem{Oliveira Tari}
Oliveira R D S, Tari F, ``On pairs of regular foliations of the plane", Hokkaido Math.Journal, 31, 523-537, 2001. 

\bibitem{Belitskii Kopanskii}
G R Belitskii, A Ya Kopanskii, ``Equivariant Sternberg Theorem", J. Dyn.Diff.Equations,  14, 349--367, 2002. 

\bibitem{Oliveira}
Oliveira R D S, ``Families of pairs of hamiltonian vector fields in the plane", in ``Real and Complex Singularities", Contemporary Mathematics, vol. 354, AMS, 2004.

\bibitem{Zhitomirskii 1}
Zhitomirskii M, ``Curves in a foliated plane", Proc. Steklov Inst. Mathematics, 259, 281-293, 2007.

\bibitem{Martins Tari}
Martins L F, Tari F, ``Vector fields tangent to foliations", Proceedings of the Edinburgh Mathematical Society, 51(3), 765-778, 2008. 

\bibitem{Kourliouros 1}
Kourliouros K, ``Local classification of constrained hamiltonian systems on 2--manifolds", J.Singularities, 6, 98-111, 2012.

\end{thebibliography}
\end{document}